\documentclass[10pt]{amsart}

\usepackage[T1]{fontenc}
\usepackage{lmodern}
\usepackage{amsfonts}
\usepackage{amsxtra}
\usepackage{amssymb}
\usepackage{mathtools}
\usepackage[margin=0.9in]{geometry}
\usepackage{xcolor}
\definecolor{cite}{rgb}{0.30,0.60,1.00}
\definecolor{url}{rgb}{0.00,0.00,0.80}
\definecolor{link}{rgb}{0.40,0.10,0.20}
\usepackage[
  colorlinks,
  linkcolor=link,
  urlcolor=url,
  citecolor=cite,
  pagebackref,
  breaklinks
]{hyperref}
\usepackage{enumitem}
\usepackage{microtype}

\allowdisplaybreaks
\numberwithin{equation}{section}
\setlist[enumerate]{leftmargin=2.2em}

\theoremstyle{plain}
\newtheorem{proposition}{Proposition}[section]
\newtheorem{theorem}[proposition]{Theorem}
\newtheorem{lemma}[proposition]{Lemma}

\theoremstyle{definition}

\theoremstyle{remark}

\DeclareMathOperator{\Rep}{Rep}
\DeclareMathOperator{\Stab}{Stab}

\DeclareMathOperator{\Lie}{Lie}
\DeclareMathOperator{\For}{For}

\newcommand{\CC}{\mathbb C}
\newcommand{\RR}{\mathbb R}
\newcommand{\ZZ}{\mathbb Z}
\newcommand{\fg}{\mathfrak g}
\newcommand{\fm}{\mathfrak m}
\newcommand{\fq}{\mathfrak q}
\newcommand{\fr}{\mathfrak r}
\newcommand{\fp}{\mathfrak p}
\newcommand{\fu}{\mathfrak u}
\newcommand{\fh}{\mathfrak h}
\newcommand{\vG}{{}^{\vee}G}
\newcommand{\vB}{{}^{\vee}B}
\newcommand{\vT}{{}^{\vee}T}
\newcommand{\vL}{{}^{\vee}L}
\newcommand{\vM}{{}^{\vee}M}
\newcommand{\vQ}{{}^{\vee}Q}
\newcommand{\vU}{{}^{\vee}U}
\newcommand{\vR}{{}^{\vee}R}
\newcommand{\vK}{{}^{\vee}K}
\newcommand{\vP}{{}^{\vee}P}
\newcommand{\vfg}{{}^{\vee}\fg}
\newcommand{\vfm}{{}^{\vee}\fm}
\newcommand{\vfr}{{}^{\vee}\fr}
\newcommand{\vfp}{{}^{\vee}\fp}
\newcommand{\vfu}{{}^{\vee}\fu}
\newcommand{\vtheta}{{}^{\vee}\theta}
\newcommand{\cF}{\mathcal F}

\newcommand{\OO}{\mathcal O}
\newcommand{\alg}{\mathrm{alg}}
\newcommand{\mic}{\mathrm{mic}}
\newcommand{\SSupp}{\operatorname{SS}}

\title[Arthur-packet support equality]
{The Arthur-Packet Support Equality for Real Reductive Groups}

\author{Jiawei Yang}
\address{School of Mathematical Sciences, Xiamen University, Xiamen 361005, China}
\email{yangjw@stu.xmu.edu.cn}

\begin{document}

\begin{abstract}
Let $\psi$ be a real Arthur parameter and let $\psi_2$ be the unipotent
parameter in a fixed Jordan decomposition. Adams--Ionov--Mason-Brown--Vogan
proved that the microlocal packet of $\psi$ is contained in the support of
the two-step Jordan induction of the packet of $\psi_2$, and conjectured
equality. We prove the reverse inclusion. The argument first passes to a
sufficiently positive translate, where the relevant connected components
of the ABV spaces have a common flag-variety model. Perverse d\'evissage reduces the problem to a
singular-support implication, which follows from the incidence calculation
underlying the AIMV microstalk comparison. Coherent continuation for
individual packet members then returns the result to the original parameter.
\end{abstract}

\maketitle

\section{Introduction}

Let $(G^\Gamma,\mathcal W)$ be an extended group, let $(\vG^\Gamma,\mathcal S)$ be a corresponding
$E$-group with second invariant $z$, and let $\psi\colon W_{\RR}\times SL_2(\CC)\longrightarrow \vG^\Gamma$
be an Arthur parameter; then $\psi|_{W_{\RR}}$ has bounded image and
$\psi|_{\{1\}\times SL_2(\CC)}$ is algebraic. Write
$\psi_{\mathrm{alg}}=\psi|_{\{1\}\times SL_2(\CC)}$. Its Arthur
$\mathfrak{sl}_2$-triple is
\[
 e=d\psi_{\mathrm{alg}}\begin{pmatrix}0&1\\0&0\end{pmatrix},
 \qquad
 f_A=d\psi_{\mathrm{alg}}\begin{pmatrix}0&0\\1&0\end{pmatrix},
 \qquad
 h=d\psi_{\mathrm{alg}}\begin{pmatrix}1&0\\0&-1\end{pmatrix}.
\]

Fix a Jordan decomposition of $\psi$ in the sense of
\cite[Definition~4.1.1]{AIMV26}, with Levi groups and parabolics
$M_2\subset M_1\subset G$, $Q_2=M_2U_2\subset M_1$, and
$Q_1=M_1U_1\subset G$.
On the dual side,
$\vQ_2=\vM_2\vU_2\subset\vM_1$. The parameter viewed in
$\vM_2^\Gamma$ is the unipotent parameter $\psi_2$, and the two-step
Jordan operator is $R_{\fq_1}^{\fg}R_{\fq_2}^{\fm_1}$, where the
right- and left-hand operations are cohomological and real parabolic
induction, respectively, in the projective normalization of
\cite[Sections~2.5--2.6]{AIMV26}.

For a parabolic $Q=LU$, let $2\rho_{\fu}$ be the determinant character
of $L$ on $\fu$, and let $
 (2\rho_{\fu})^\vee\colon\CC^\times\longrightarrow\vL$
be the corresponding algebraic cocharacter on the dual side. Then $
 z(\rho_{\fu}):=(2\rho_{\fu})^\vee(-1)\in Z(\vL)^{\theta_Z}$
is the central element of order dividing two
\cite[(2.5.1)--(2.5.2)]{AIMV26}.
Put $z_1:=zz(\rho_{\fu_1})$ and
$z_2:=z_1z(\rho_{\fu_2})=zz(\rho_{\fu_1})z(\rho_{\fu_2})$. The extended-group
induction maps of \cite[(2.6.1)]{AIMV26} have projective types\footnote{We
follow the conventions of \cite[Sections~2.5--2.6]{AIMV26}. The target
superscript in its v1 display (4.5.1) should be $z$, consistently with
its Corollary~4.5.1(ii).}
\begin{equation}\label{eq:projective-type-ledger}
 \begin{split}
 R_{\fq_2}^{\fm_1}\colon
 K\Pi^{z_2}(M_2^\Gamma)&\longrightarrow K\Pi^{z_1}(M_1^\Gamma),\\
 R_{\fq_1}^{\fg}\colon
 K\Pi^{z_1}(M_1^\Gamma)&\longrightarrow K\Pi^z(G^\Gamma).
\end{split}
\end{equation}

For a real parabolic $Q(\RR)=L(\RR)U(\RR)$, this convention means
\begin{equation}\label{eq:real-induction-normalization}
 R_{\fq}^{\fg}(V)
 =\operatorname{Ind}_{Q(\RR)}^{G(\RR)}
   \left(V\otimes\frac{\rho_{\fu}}{|\rho_{\fu}|}\right),
\end{equation}
where $\operatorname{Ind}$ is the usual normalized real parabolic
induction, $\rho_{\fu}$ is the canonical square-root character of
$2\rho_{\fu}$ on the canonical cover, and
$|\rho_{\fu}|=|2\rho_{\fu}|^{1/2}$ is pulled back from $L(\RR)$;
see \cite[Section~2.5.1]{AIMV26}.
The fixed phase character $\rho_{\fu}/|\rho_{\fu}|$ has zero
infinitesimal character and projective type $z(\rho_{\fu})$.
Thus it changes type $zz(\rho_{\fu})$ to type $z$, which is then
preserved by $\operatorname{Ind}$. This accounts for the second
type change in \eqref{eq:projective-type-ledger} without an additional
infinitesimal-character shift.

For a virtual representation $V=\sum_\pi n_\pi\pi$ of finite length in the
Grothendieck group, put $[V]=\{\pi:n_\pi\neq0\}$, and take unions for sets
of virtual representations.
Adams--Ionov--Mason-Brown--Vogan prove the inclusion
\cite[Theorem~4.6.1(iii)]{AIMV26}
\begin{equation}\label{eq:intro-known}
  \Pi^z(G^\Gamma)^{\mic}_{\psi}
  \subseteq
  \left[
  R_{\fq_1}^{\fg}R_{\fq_2}^{\fm_1}
  \Pi^{zz(\rho_{\fu_1})z(\rho_{\fu_2})}(M_2^\Gamma)^{\mic}_{\psi_2}
  \right]
\end{equation}
and formulate equality as \cite[Conjecture~4.6.5]{AIMV26}. The reverse
inclusion is not formal because signed contributions from distinct source
members may cancel.

\begin{theorem}\label{thm:main}
Let $(G^\Gamma,\mathcal W)$ and $(\vG^\Gamma,\mathcal S)$ be as above,
with second invariant $z$. Let $\psi$ be a real Arthur parameter and fix
a Jordan decomposition in the sense of
\cite[Definition~4.1.1]{AIMV26}. Set
$z_1=zz(\rho_{\fu_1})$ and $z_2=z_1z(\rho_{\fu_2})$. Then
\begin{equation}\label{eq:main}
  \Pi^z(G^\Gamma)^{\mic}_{\psi}
  =
  \left[
  R_{\fq_1}^{\fg}R_{\fq_2}^{\fm_1}
  \Pi^{z_2}(M_2^\Gamma)^{\mic}_{\psi_2}
  \right].
\end{equation}
\end{theorem}

Here is the main geometric step. After a sufficiently positive translation,
we restrict to the genuine connected components containing the Arthur
geometric points. Lemma~\ref{lem:fiber-model} proves that this restriction
commutes with global convolution. These components have the common
correspondence of flag varieties
$X_g\xleftarrow{q_g}Z\xrightarrow{q_f}X_f$, under which
$\operatorname{Fib}_g\circ I_\psi\simeq
(q_g)_!q_f^*\circ\operatorname{Fib}_f$. The Arthur covectors are
$\nu_f=(x_f,e)$ and $\nu_g=(x_g,e)$. For every
$P\in\mathcal P_f^\psi$, the perverse heart on the source component,
we prove
\begin{equation}\label{eq:intro-propagation}
  \nu_g\in\SSupp(I_\psi P)
  \quad\Longrightarrow\quad
  \nu_f\in\SSupp(P)
  \qquad(P\in\mathcal P_f^\psi).
\end{equation}
If a visible simple perverse sheaf occurs in $[I_\psi P]_{K_0}$, perverse
d\'evissage gives $\nu_g\in\SSupp(I_\psi P)$. The incidence calculation underlying
\cite[Lemma~4.6.2]{AIMV26}, together with the three-parabolics theorem,
then yields \eqref{eq:intro-propagation}. A coherent-family identity for each
source member carries the result back from the good-range translate.

\section{Background from ABV and AIMV}

\subsection{ABV packets and Vogan duality}

The subscripts $f$ and $g$ refer to the original endpoint and a sufficiently
positive translate, respectively. Write $\psi_f=\psi$ and
$\psi_g=\psi^N$, where $\psi^N$ is constructed in the next subsection.
Let $\OO_a\subset\vfg$ be the corresponding orbit of infinitesimal character,
and let $Y$ denote the geometric $\vG$-orbit. The corresponding blocks at the two
endpoints have the same $Y$ by Lemma~\ref{lem:fiber-model}.
Put $\mathcal X_a=X(\vG^\Gamma,\OO_a,Y)$, and let
$\mathcal D_a=D^bC^{\vG^{\alg}}(\mathcal X_a)$ and
$\mathcal P_a=P^{\vG^{\alg}}(\mathcal X_a)$
be the bounded equivariant constructible derived
category on this $Y$-block and its perverse heart, respectively.
The $Y$-block need not be connected. Let
$j_a:\mathcal X_a^\psi\hookrightarrow\mathcal X_a$ be the open and
closed connected component containing the Arthur geometric point of
$\psi_a$, and write $\mathcal D_a^\psi$ and $\mathcal P_a^\psi$ for
its equivariant derived category and perverse heart. Extension by zero
identifies these with direct summands of $\mathcal D_a$ and
$\mathcal P_a$. The
corresponding abelian category of equivariant perverse sheaves is that of
\cite[Definition~7.7]{ABV92}. Let $\Xi_a$ be the set of equivalence classes
of complete geometric parameters with infinitesimal character $\OO_a$,
allowing all $Y$-blocks. A class $\xi\in\Xi_a$ is represented by a $\vG$-orbit
$S_\xi$ and an irreducible equivariant local system $V_\xi$ on it
\cite[Definition~7.6]{ABV92}. Write $P_a(\xi)$ for its perverse
intermediate extension and $\pi_a(\xi)$ for the corresponding irreducible
canonical projective representation of type $z$
\cite[(7.10)(d), Theorem~10.4, and Definition~10.8]{ABV92}.
The notation $P_a(\xi)$ and the pairing below are also used on the full
parameter space, the disjoint union of the $Y$-blocks;
$P_a(\xi)\in\mathcal P_a$ precisely when its support lies in the fixed
block. The Vogan duality
pairing of \cite[Theorem~15.12]{ABV92} is diagonal:
\begin{equation}\label{eq:vogan}
 \langle\pi_a(\xi),P_a(\xi')\rangle_a
 =
 e(\xi)(-1)^{d(\xi)}\delta_{\xi,\xi'}.
\end{equation}
Here $e(\xi)$ is the Kottwitz sign and $d(\xi)=\dim_{\CC}S_\xi$.

For $C\in\mathcal D_a$, its class in the Grothendieck group of
$\mathcal P_a$ is
\[
 \begin{aligned}
 [C]_{K_0}
   &:=\sum_j(-1)^j[{}^p\mathcal H^j(C)] \\
   &=\sum_\xi\left(\sum_j(-1)^j
      [{}^p\mathcal H^j(C):P_a(\xi)]\right)[P_a(\xi)].
 \end{aligned}
\]
Here $[Q:P_a(\xi)]$ denotes the Jordan--H\"older multiplicity of
$P_a(\xi)$ in a finite-length perverse sheaf $Q$.

Let $\nu_a$ be the regular Arthur covector. The normalized microstalk
functor $\chi_{\psi_a}^{\mic}\colon\mathcal P_a\to
\Rep(A_{\psi_a}^{\mic})$ is exact by \cite[Theorem~24.8(c)]{ABV92}.
It is zero on the other connected components and is extended by zero
to the other $Y$-blocks. At the regular Arthur
covector, the ABV microstalk functor gives the support criterion
\begin{equation}\label{eq:packet-criterion}
 \chi_{\psi_a}^{\mic}(P_a(\xi))\neq0
 \quad\Longleftrightarrow\quad
 \nu_a\in\SSupp(P_a(\xi)).
\end{equation}
This is \cite[Theorem~24.8(a),(c), Corollary~24.9(a), and
Proposition~19.12(c)]{ABV92}. We call $P_a(\xi)$ \emph{visible at $\nu_a$}
when \eqref{eq:packet-criterion} holds. By
\cite[Definition~22.6]{ABV92}, the Arthur packet of $\psi_a$ is
$\Pi^{z}(G^\Gamma)^{\mic}_{\psi_a}
=\{\pi_a(\xi):\chi_{\psi_a}^{\mic}(P_a(\xi))\neq0\}$.

\subsection{Positive translation, projective types, and the good-range bijection}

Let $\fu_2=\Lie(U_2)$ and $\vfu_2=\Lie(\vU_2)$ be the nilradicals
of the corresponding Jordan parabolics. Fix dual maximal tori
$H\subset M_2$ and $\vT\subset\vM_2$, and put $\fh=\Lie(H)$.
Write $\Delta(\fu_2,\fh)$ and $\Delta(\vfu_2)$ for the respective
nilradical root sets. Define the algebraic character on the original group by
\[
 \chi(m):=\det\!\left(\operatorname{Ad}(m)|_{\fu_2}\right),
 \qquad
 \mu_\chi=\sum_{\beta\in\Delta(\fu_2,\fh)}\beta
          =2\rho_{\fu_2}\in X^*(M_2)\subset X^*(H).
\]
Root datum duality $X^*(H)=X_*(\vT)$ identifies this weight with the
central algebraic cocharacter
$\chi^\vee\colon\CC^\times\to Z(\vM_2)$. In dual-root notation,
its class in $X_*(\vT)$ is $\sum_{\alpha\in\Delta(\vfu_2)}\alpha^\vee$,
the sum of the corresponding coroots.
For a sufficiently large positive integer $N$ (see
Lemma~\ref{lem:continuation-Levi}), let $\psi_2^N$ be the corresponding
central twist of
$\psi_2$, and let $\psi^N$ be its composite with
$\vM_2^\Gamma\hookrightarrow\vG^\Gamma$. Then we have the formula
\cite[(4.6.1)]{AIMV26} 
\begin{equation}\label{eq:packet-twist}
 \Pi^{z_2}(M_2^\Gamma)^{\mic}_{\psi_2^N}
 =
 \{\sigma\otimes\chi^{2N}:
   \sigma\in\Pi^{z_2}(M_2^\Gamma)^{\mic}_{\psi_2}\}.
\end{equation}
Put $\sigma^N=\sigma\otimes\chi^{2N}$. Let $\lambda_{2,f}$ and
$\lambda_{2,g}$ be the infinitesimal characters of $\psi_2$ and
$\psi_2^N$, and let $\lambda_f$ and $\lambda_g$ be their images for
$G^\Gamma$. With this determinant-character convention, the calculation in
\cite[proof of Theorem~4.6.1, immediately after (4.6.1)]{AIMV26} gives
\begin{equation}\label{eq:translation}
 \lambda_{2,g}-\lambda_{2,f}=2N\mu_\chi,
 \qquad
 \lambda_g-\lambda_f=2N\mu_\chi=4N\rho_{\fu_2}.
\end{equation}
In the geometric formulas, we regard $\lambda_f$ and $\lambda_g$ as elements
of $\Lie(\vT)\subset\vfg$.

For a semisimple element $\lambda\in\vfg$, write $
 \vfg_n(\lambda)
 =
 \{X\in\vfg:[\lambda,X]=nX\},
 \mathfrak n(\lambda)
 =
 \bigoplus_{n\in\ZZ_{>0}}\vfg_n(\lambda).
$
Following \cite[Section~3.1]{AIMV26}, the canonical flat through $\lambda$
is the affine subspace $\cF(\lambda)=\lambda+\mathfrak n(\lambda)\subset\vfg$.
For the two endpoints, define
\begin{equation}\label{eq:canonical-flats}
 \Lambda_f:=\cF(\lambda_f),
 \qquad
 \Lambda_g:=\cF(\lambda_g).
\end{equation}
Following \cite[Definition~4.1.2]{AIMV26}, an infinitesimal character
$\lambda\in\vfm_2$ is in good range for
$\vQ_2=\vM_2\vU_2$ if
\begin{equation}\label{eq:good-range-condition}
 \operatorname{Re}\alpha(\lambda)>0,
 \qquad
 \text{for every }\alpha\in\Delta(\vfu_2).
\end{equation}

Set $
 \kappa:=2\mu_\chi=4\rho_{\fu_2}\in X^*(H),
 \lambda_g=\lambda_f+N\kappa$.

\begin{lemma}\label{lem:continuation-Levi}
The positive integer $N$ may be chosen so that
\eqref{eq:good-range-condition} holds for
$\lambda=\lambda_g$ and there is a Levi subgroup
$L_N\subset G$ containing $H$ and a parabolic subalgebra $
 \fq_N=\mathfrak l_N\oplus\mathfrak u_N,
 \mathfrak l_N=\Lie(L_N)$
with the following properties:
\begin{enumerate}
\item $\lambda_g$ is in the good range for $\fq_N$ in the sense of
      \cite[Definition~2.5.3]{AIMV26};
\item $\lambda_g-\lambda_f=N\kappa$
      belongs to $X^*(L_N)$.
\end{enumerate}
\end{lemma}

\begin{proof}
For $\alpha\in\Delta(\fg,\fh)$, put
$a_\alpha=\operatorname{Re}\langle\lambda_f,\alpha^\vee\rangle$ and
$b_\alpha=\langle\kappa,\alpha^\vee\rangle\in\ZZ$.
Under root datum duality, for every $\alpha\in\Delta(\vfu_2)$ and
the corresponding $\beta\in\Delta(\fu_2,\fh)$ one has
$\alpha(\mu_\chi)=\langle2\rho_{\fu_2},\beta^\vee\rangle>0$
by \cite[Lemma~2.5.4]{AIMV26}. Hence
\eqref{eq:good-range-condition} holds for all sufficiently large $N$.
Since there are only finitely many roots, we
may at the same time avoid the finitely many values of $N$ for which
$a_\alpha+Nb_\alpha=0$ for at least one root $\alpha$ with
$b_\alpha\neq0$. Thus
\begin{equation}\label{eq:generic-N}
 a_\alpha+Nb_\alpha\neq0
 \qquad
 \text{whenever }b_\alpha\neq0.
\end{equation}
Define
\[
\begin{aligned}
 \mathfrak l_N
 &=
 \fh\oplus
 \bigoplus_{a_\alpha+Nb_\alpha=0}\fg_\alpha,
 \qquad
 \mathfrak u_N = \bigoplus_{a_\alpha+Nb_\alpha>0}\fg_\alpha,
 \qquad
 \fq_N=\mathfrak l_N\oplus\mathfrak u_N.
\end{aligned}
\]
Choose a positively definite $W(G,H)$-invariant form on $X^*(H)_\RR$ and use
it to identify $X^*(H)_\RR$ with $X_*(H)_\RR$. Since
$\alpha^\vee=2\alpha/(\alpha,\alpha)$ under this identification, the signs
of $\operatorname{Re}\langle\lambda_g,\alpha^\vee\rangle$ give the usual
parabolic root decomposition. The roots for which
$a_\alpha+Nb_\alpha=0$ form the Levi subsystem annihilated by
$\operatorname{Re}\lambda_g$. The zero and positive root spaces therefore
define a parabolic subalgebra $\fq_N$ with Levi factor $\mathfrak l_N$.
The same root subsystem determines a connected algebraic Levi subgroup
$L_N\subset G$ containing $H$, with Lie algebra $\mathfrak l_N$.
For every root $\alpha$ of $\mathfrak u_N$, one has
$\operatorname{Re}\langle\lambda_g,\alpha^\vee\rangle
=a_\alpha+Nb_\alpha>0$, so $\lambda_g$ is in the good range of
\cite[Definition~2.5.3]{AIMV26}.

If $\alpha$ is a root of $\mathfrak l_N$, then
$a_\alpha+Nb_\alpha=0$, and \eqref{eq:generic-N} forces
$b_\alpha=0$. Thus the integral weight $\kappa$ vanishes on every
coroot of the derived group of $L_N$. Since $L_N$ is connected with maximal
torus $H$, restriction to $H$ identifies $X^*(L_N)$ with the weights
$\nu\in X^*(H)$ satisfying $\langle\nu,\alpha^\vee\rangle=0$ for every root
$\alpha$ of $L_N$.
It follows that $\kappa\in X^*(L_N)$, and hence
$\lambda_g-\lambda_f=N\kappa\in X^*(L_N)$.
\end{proof}

Fix $N$ satisfying the lemma and define
$T=T_{L_N}(\lambda_g\to\lambda_f)$ to be the coherent continuation operator. It maps
$K\Pi_{\lambda_g}^z(G^\Gamma)$ to $K\Pi_{\lambda_f}^z(G^\Gamma)$, and the
coherent continuations on the individual
$(\fg,K_\gamma)$-module Grothendieck groups, defined in
\cite[Definition~2.10.3]{AIMV26}, are summed over strong real forms; see
\eqref{eq:strong-real-form-sum}--\eqref{eq:summandwise-continuation}.

Write $
 R=R_{\fq_1}^{\fg}R_{\fq_2}^{\fm_1},
 \widetilde R
 =e(G^\Gamma)R_{\fq_1}^{\fg}R_{\fq_2}^{\fm_1}e(M_2^\Gamma).
$
Here $e(G^\Gamma)$ and $e(M_2^\Gamma)$ are the Kottwitz signs of
\cite[Theorem~4.2.3]{AIMV26}. Define the restrictions of $R$ and
$\widetilde R$ by
\[
\begin{aligned}
 R_N,\widetilde R_N\colon
 K\Pi_{\lambda_{2,g}}^{z_2}(M_2^\Gamma)
 &\longrightarrow K\Pi_{\lambda_g}^{z}(G^\Gamma),\\
 R_0,\widetilde R_0\colon
 K\Pi_{\lambda_{2,f}}^{z_2}(M_2^\Gamma)
 &\longrightarrow K\Pi_{\lambda_f}^{z}(G^\Gamma).
\end{aligned}
\]

For the translated parameters $\psi_2^N$ and $\psi^N$,
\cite[Remark~4.5.2]{AIMV26} gives a bijection $
 R_N\colon
 \Pi^{z_2}(M_2^\Gamma)^{\mic}_{\psi_2^N}
 \longrightarrow
 \Pi^z(G^\Gamma)^{\mic}_{\psi^N}
$.
This map sends each irreducible, up to sign, to a single nonzero irreducible.
Thus, for every $\sigma^N$ there are unique
$\tau_g\in\Pi^z(G^\Gamma)^{\mic}_{\psi^N}$ and signs
$\beta_{\sigma,N},\alpha_{\sigma,N}\in\{\pm1\}$ such that
\begin{equation}\label{eq:good-range}
 R_N\sigma^N=\beta_{\sigma,N}\tau_g,
 \qquad
 \widetilde R_N\sigma^N=\alpha_{\sigma,N}\tau_g.
\end{equation}
The second equality follows because the Kottwitz operators are diagonal.

\subsection{Connected components and the flag-variety model}

\begin{lemma}\label{lem:fiber-model}
The two endpoints have the same geometric $y$-coordinate and the same
$t=\exp(2\pi i\lambda_a)=y^2$. Put
\[
 \widehat{\vR}=Z_{\vG}(t),\qquad
 \vR=\widehat{\vR}^{\circ},\qquad
 \widehat{\vK}=Z_{\vG}(y),\qquad
 \vK=\widehat{\vK}\cap\vR,
 \qquad \vtheta=\operatorname{Ad}(y)|_{\vR}.
\]
For $a\in\{f,g\}$, let $\vP_a=\Stab_{\vG}(\Lambda_a)$, a connected
parabolic subgroup of $\vR$, let $\vfu_a$ be the nilradical of
$\Lie(\vP_a)$, and put
$X_a=\vR/\vP_a$. The full $Y$-block and its Arthur connected component are
\begin{equation}\label{eq:full-and-Arthur-models}
 \mathcal X_a\simeq
 \vG\times^{\widehat{\vK}}(\widehat{\vR}/\vP_a),
 \qquad
 \mathcal X_a^\psi\simeq\vG\times^{\vK}X_a.
\end{equation}
Write $\vK^{\alg}$ for the full inverse image of $\vK$ under
$\vG^{\alg}\to\vG$. For the inclusion $i_a:X_a\hookrightarrow
\mathcal X_a^\psi$ of the fiber over $1\vK$, the normalized restriction
\[
 \operatorname{Fib}_a=i_a^*[-d]:\mathcal D_a^\psi
 \xrightarrow{\sim}D^bC^{\vK^{\alg}}(X_a),
 \qquad d=\dim_{\CC}(\vG/\vK),
\]
is an equivalence whose restriction to perverse hearts is an equivalence.
It identifies the Arthur conormal datum and microstalk with those at
$\nu_a=(x_a,e)$, where $x_a=1\vP_a$.

Let $Z=\vR/(\vP_g\cap\vP_f)$ with natural projections $q_a:Z\to X_a$,
and put $I_{\mathrm{fib}}=(q_g)_!q_f^*$. Let
$I_\psi:\mathcal D_f^\psi\to\mathcal D_g^\psi$ be the functor transported
from $I_{\mathrm{fib}}$ by the two normalized equivalences.
The global convolution $I=I(\Lambda_f\to\Lambda_g)$ satisfies
\begin{equation}\label{eq:component-convolution}
 j_g^* I\simeq I_\psi j_f^*,
 \qquad
 Ij_{f!}\simeq j_{g!}I_\psi,
 \qquad
 \operatorname{Fib}_g I_\psi\simeq
 I_{\mathrm{fib}}\operatorname{Fib}_f.
\end{equation}
\end{lemma}

\begin{proof}
Let $\varphi_a=\varphi_{\psi_a}$ be the Langlands parameter associated with
$\psi_a$. By \cite[Proposition~3.1.2]{AIMV26}, the first coordinate of the
corresponding geometric parameter is
$y_a=\exp(\pi i\lambda_a)\varphi_a(j)$.
The positive twist is trivial on $j$ and on the Arthur $SL_2(\CC)$;
see the proof of \cite[Theorem~4.6.1, before (4.6.1)]{AIMV26}.
Hence $\varphi_g(j)=\varphi_f(j)$, and $e$ is the same at both endpoints.
The duality of the fixed tori identifies $\mu_\chi$ with
$d\chi^\vee(1)\in\Lie(\vT)$. Therefore
\[
 \exp(\pi i(\lambda_g-\lambda_f))=\chi^\vee(e^{2\pi iN})=1,
 \qquad
 \exp(2\pi i(\lambda_g-\lambda_f))=\chi^\vee(e^{4\pi iN})=1.
\]
Thus $y_g=y_f=:y$ and $t_g=t_f=:t$, so the full centralizers, as well
as their indicated subgroups, coincide.

By \cite[Lemma~6.3(d) and Proposition~6.5(c)]{ABV92}, the full stabilizer
$\vP_a$ of $\Lambda_a$ is connected and contained in $\vR$.
The first model in \eqref{eq:full-and-Arthur-models} is
\cite[Proposition~6.16]{ABV92}, using the full centralizers.
The connected components of $\widehat{\vR}/\vP_a$ are indexed by
$\widehat{\vR}/\vR$. Consequently the connected components of
$\mathcal X_a$ are indexed by the finite double-coset set
\begin{equation}\label{eq:component-labels}
 \mathfrak C=\widehat{\vK}\backslash\widehat{\vR}/\vR,
 \qquad [g,r\vP_a]\longmapsto\widehat{\vK}r\vR.
\end{equation}
Indeed, these labels define open and closed subsets; the subset with
label $\widehat{\vK}\vR$ is
\[
 \vG\times^{\widehat{\vK}}
       (\widehat{\vK}\vR/\vP_a)
 \simeq\vG\times^{\vK}(\vR/\vP_a),
\]
which is connected because $\vG$ and $\vR/\vP_a$ are connected.
The same argument after conjugating by $r$ applies to every label.
The identity label contains $(y,\Lambda_a)$ and hence gives
$\mathcal X_a^\psi$. This is the connected-component description of
\cite[(6.23) and Proposition~6.24]{ABV92}.
In particular, $\vK=\widehat{\vK}\cap\vR$ is not being replaced by
$\vK^\circ$.

The induced-space equivalence is
\cite[Proposition~7.14]{ABV92}; the convention on $\vK^{\alg}$ is that
of \cite[(5.10) and (8.4)]{ABV92}.
Every orbit in $\vG\times^{\vK}X_a$ is $\vG\times^{\vK}S$ for a
$\vK$-orbit $S\subset X_a$, of dimension $d+\dim_{\CC}S$.
A local analytic section of $\vG\to\vG/\vK$ identifies the orbit
stratification with the product of a smooth $d$-dimensional base and
that on $X_a$. Thus $i_a^*[-d]$ gives the perverse normalization.
More generally, an equivariant bounded complex is locally the pullback
of its fiber restriction along this product projection. The smooth
inverse-image formula for singular support therefore identifies its
corresponding conormal directions, including for complexes that are
not perverse.
Under this identification, the Arthur covector is $(x_a,e)$ by
\cite[Propositions~3.7.1 and~3.7.3]{AIMV26}.
The point and conormal stabilizers are unchanged: an element of
$\widehat{\vK}$ fixing $x_a$ lies in $\vP_a\subset\vR$, hence in
$\vK$, and the same holds in their full inverse images in
$\vG^{\alg}$. Together with the normal-slice identification, this
preserves the Arthur microstalk and its component-group representation;
see \cite[Propositions~7.14 and~21.3]{ABV92}.

To check convolution, retain the full correspondence
\begin{equation}\label{eq:full-incidence}
 \widehat{\mathcal Z}
 :=\vG\times^{\widehat{\vK}}
       \bigl(\widehat{\vR}/(\vP_g\cap\vP_f)\bigr),
 \qquad Q_a:\widehat{\mathcal Z}\longrightarrow\mathcal X_a,
 \qquad I=(Q_g)_!Q_f^*.
\end{equation}
This realizes the convolution of \cite[Section~3.5]{AIMV26} on all
connected components. The relation with the global duality theorem is
spelled out in the next subsection.
Both projections send $[g,r(\vP_g\cap\vP_f)]$ to a point with label
$\widehat{\vK}r\vR$ in \eqref{eq:component-labels}. Hence
\begin{equation}\label{eq:equal-component-preimages}
 Q_g^{-1}(\mathcal X_g^\psi)
 =Q_f^{-1}(\mathcal X_f^\psi)
 =:\mathcal Z^\psi
 \simeq\vG\times^{\vK}Z.
\end{equation}
Let $j_Z:\mathcal Z^\psi\hookrightarrow\widehat{\mathcal Z}$ be the
open and closed inclusion and $Q_a^\psi$ the restricted projections.
Open and closed base change gives, for every $C\in\mathcal D_f$,
\[
 \begin{aligned}
 j_g^*(Q_g)_!Q_f^*C
 &\simeq(Q_g^\psi)_!j_Z^*Q_f^*C\\
 &\simeq(Q_g^\psi)_!(Q_f^\psi)^*j_f^*C.
 \end{aligned}
\]
The same equality of inverse images gives
$I j_{f!}\simeq j_{g!}(Q_g^\psi)_!(Q_f^\psi)^*$.
Under the induced-space equivalences, the restricted correspondence is
$X_g\leftarrow Z\to X_f$; the normalization $[-d]$ is the same on both
sides. Thus $(Q_g^\psi)_!(Q_f^\psi)^*$ is $I_\psi$, proving
\eqref{eq:component-convolution}. This base-change argument requires
no properness of $Q_g$.
\end{proof}

We use the same symbol $\nu_a$ for the corresponding Arthur conormal
data on $\mathcal X_a^\psi$ and $X_a$; the ambient space in a
singular-support statement is the space of the displayed complex.

\subsection{Continuation--convolution duality}

Here global convolution means \eqref{eq:full-incidence} on each
$Y$-block, and the direct sum over $Y$ on the full parameter spaces.

\begin{proposition}[Global continuation--convolution duality]
\label{prop:global-duality}
For the endpoints chosen above, put $I_K=K_0(I)$. The continuation
$T=T_{L_N}(\lambda_g\to\lambda_f)$ and the full convolution
$I(\Lambda_f\to\Lambda_g)$ are transpose under the Vogan pairings:
\begin{equation}\label{eq:duality}
 \langle Tv,x\rangle_f
 =
 \langle v,I_Kx\rangle_g
\end{equation}
for every representation Grothendieck class $v$ at the $g$-endpoint
and every sheaf Grothendieck class $x$ at the $f$-endpoint, allowing
all $Y$-blocks and all their connected components.
\end{proposition}

\begin{proof}
Lemma~\ref{lem:fiber-model} realizes the full correspondence on each
$Y$-block and shows that its two projections preserve the labels
\eqref{eq:component-labels}. Thus its sheaf functor decomposes over
the corresponding connected components. The diagonal Vogan pairing
\eqref{eq:vogan} makes the associated Grothendieck summands orthogonal.
In the regular case, the geometric Hecke action is the direct sum of
the actions on the connected flag models
\cite[(8.9), Propositions~16.13 and~17.16]{ABV92}.
When $\vP'\subseteq\vP$, convolution is pullback along the global
flag projection of \cite[Proposition~8.8]{ABV92}; its parameter map
agrees with translation by \cite[Proposition~16.6]{ABV92}.
The composition calculation in \cite[Lemma~3.5.3]{AIMV26} uses fibers
of the form $(\vP\cap\vP')/(\vP_r\cap\vP')$, which are unchanged
on passing to the full centralizer or to any connected component.
Consequently the regularization argument proving
\cite[Theorem~3.6.3]{AIMV26} applies to this direct sum, with the
same Euler-characteristic factors and no additional component factor.
This does not presuppose that representation-theoretic continuation
preserves the connected-component decomposition.

With the canonical flats defined in
\eqref{eq:canonical-flats}, we apply
\cite[Theorem~3.6.3]{AIMV26} with $L=L_N$, $\lambda'=\lambda_g$,
$\lambda=\lambda_f$, $\Lambda=\Lambda_f$, and $\Lambda'=\Lambda_g$.
The two hypotheses \textup{(GR1)} and \textup{(GR2)} are
Lemma~\ref{lem:continuation-Levi}(1) and (2), respectively. The theorem,
realized on the full direct sum as above, gives \eqref{eq:duality}.
We use the sign convention of \cite[Theorem~3.6.3]{AIMV26} in
\eqref{eq:duality}.
\end{proof}

\subsection{The three-parabolics identity}

We shall use the following specialization of the three-parabolics theorem
\cite[Theorem~A.2.1]{AIMV26}.

\begin{lemma}\label{lem:three-parabolics}
For the translation datum $(\Lambda_f,\Lambda_g)$ constructed above, with
$\lambda_g-\lambda_f=4N\rho_{\fu_2}$ and $N\gg0$, the set of
$r\in\overline{\vP_f\vP_g}$ satisfying
$\operatorname{Ad}(r)e\in\vfu_f$ is exactly $\vP_f$.
\end{lemma}

\begin{proof}
Let $\vQ_2=\vM_2\vU_2\subset\vM_1$ be the dual Jordan parabolic and put
$\vQ=\vQ_2\cap\vR$. By \cite[Lemma~4.3.1]{AIMV26}, the pseudo-Levi of
$\vM_1$ determined by $\lambda_f$ is $\vR$, and
\cite[Lemma~4.4.1(i)]{AIMV26} shows that $\vQ$ is a parabolic subgroup of
$\vR$ with Levi factor
$\vL:=Z_{\vM_2}(\exp(2\pi i\lambda_f))^\circ$; write $\vQ=\vL\vU$.

For $a\in\{f,g\}$, write $\lambda_a=\tfrac12h+\zeta_a$, and put
$\zeta_a^{\mathbb R}=\operatorname{Re}\zeta_a$ and
$\lambda_a^{\mathbb R}=\tfrac12h+\zeta_a^{\mathbb R}$. The Arthur triple $(e,f_A,h)$ lies in $\vfm_2$. Since
$\lambda_f=\tfrac12h+\zeta_f$ and $\zeta_f$ centralizes the Arthur
$\mathfrak{sl}_2$, its $\operatorname{ad}(\lambda_f)$-weights are
$1,-1,0$. Thus the triple is fixed by $\exp(2\pi i\lambda_f)$ and lies in
$\vfr$, so $(e,f_A,h)\subset\vfr\cap\vfm_2={}^\vee\mathfrak l$. After conjugation by
$\vL$, choose $\vT\subset\vB\subset\vR$ so that $\vQ$ is standard,
$h\in\Lie(\vT)$, and $h$ is dominant for the positive roots of $\vL$.

Replacing $\lambda_a$ by $\lambda_a^{\mathbb R}$ does not change the
grading on $\vfr$. Indeed, for every root $\alpha$ of $\vR$, the definition
of $\vR$ gives $\exp(2\pi i\alpha(\lambda_a))=1$, hence
$\alpha(\lambda_a)\in\ZZ$. Since $\alpha(h)$ is real,
$\alpha(\operatorname{Im}\zeta_a)=0$. Thus
$\operatorname{Im}\zeta_a$ is central in $\vfr$. Hence
$\vfr_n(\lambda_a)=\vfr_n(\lambda_a^{\mathbb R})$ for every $n\in\ZZ$.
Since the stabilizer
$\vP_a=\Stab_{\vR}(\Lambda_a)$ has Lie algebra
$\Lie(\vP_a)=\bigoplus_{n\geq0}\vfr_n(\lambda_a)$, $\vP_a$ is also
the grading parabolic attached to $\lambda_a^{\mathbb R}$.

We verify the hypotheses of \cite[Theorem~A.2.1]{AIMV26}. First, the
preceding argument shows that $\lambda_f^{\mathbb R}$ and
$\lambda_g^{\mathbb R}$ take integral values on every root of $\vR$.
Second, Definition~4.1.1(4) of \cite{AIMV26} identifies $\psi_2$ as a
unipotent Arthur parameter for $\vM_2^\Gamma$, while
\cite[Definition~2.3.3]{AIMV26} says that
$\psi_2(\CC^\times)\subset Z(\vM_2)$. Differentiating gives
$\zeta_f\in Z(\vfm_2)$. By the positive
translation formula \cite[(4.6.1)]{AIMV26}, $\zeta_g-\zeta_f$ is the
differential of a central cocharacter of $\vM_2$. Hence
$\zeta_f^{\mathbb R}$ and
$\zeta_g^{\mathbb R}$ are central on $\vL$, and
$\alpha(h)=2\alpha(\lambda_f^{\mathbb R})\in2\ZZ$ for every root $\alpha$
of $\vL$. The grading defined by $h$ on ${}^\vee\mathfrak l$ is therefore even.

For the weakly fair condition, fix $\alpha\in\Delta(\vU,\vT)$ and
$0\neq X_\alpha\in\vfr_\alpha$. Let $\zeta_f^{0,1}$ denote the
antiholomorphic part of the differential of
$\psi|_{\CC^\times}$. On $\CC X_\alpha$, this restriction has exponents
$A=\alpha(\zeta_f)$ and
$B=\alpha(\zeta_f^{0,1})$. Single-valuedness gives $A-B\in\ZZ$, and
boundedness gives $\operatorname{Re}(A+B)=0$. The preceding argument gives
$A\in\RR$, whence $B\in\RR$ and $B=-A$. The $S^1$-weight of the root line is
therefore $m_\alpha=A-B=2A=2\alpha(\zeta_f^{\mathbb R})$. By definition,
$\vfu_2$ is the sum of the positive $S^1$-weight spaces.
Since $\vfu\subset\vfu_2$, we have $m_\alpha>0$, and hence
$\alpha(\zeta_f^{\mathbb R})>0$.

Finally, $N$ was chosen so that $\operatorname{Re}\alpha(\lambda_g)>0$ for
every root $\alpha$ of $\vfu_2$. Since $\vfu\subset\vfu_2$ and
$\lambda_g^{\mathbb R}$ gives the same grading on $\vfr$, every
$\lambda_g^{\mathbb R}$-weight on $\vU$ is strictly positive. All the
hypotheses of \cite[Theorem~A.2.1]{AIMV26} now hold for $\vR$, the
parabolic $\vQ=\vL\vU$, the Arthur triple $(e,f_A,h)$, and the grading
elements $\lambda_f^{\mathbb R},\lambda_g^{\mathbb R}$. Its conclusion is
the asserted identity.
\end{proof}

\section{Microsupport preliminaries}

\subsection{Cotangent conventions}\label{sec:cotangent-conventions}

Kashiwara--Schapira microsupport is a subset of the real cotangent bundle
of the underlying real analytic manifold. For a complex manifold $X$, we
use the convention of \cite[Section~8.5.1]{KS85}, identifying
$T_x^{*(1,0)}X$ with $T^*_{x,\RR}X$ by $\xi\mapsto2\operatorname{Re}\xi$.
The factor $2$ makes the canonical real one-form twice the real part of the
holomorphic canonical one-form. By \cite[Theorem~8.5.2]{KS85}, the
microsupport of a complex constructible sheaf is invariant under the
natural $\CC^\times$-action.
Throughout, $\SSupp$ denotes this Kashiwara--Schapira microsupport,
expressed in holomorphic cotangent coordinates through this
identification; we also call it singular support.
For equivariant perverse sheaves, the ABV characteristic-cycle
multiplicities are used at regular conormal points through their
identification with microstalk ranks
\cite[Theorem~24.8(a)]{ABV92}; at the Arthur covectors this gives
the detection criterion \eqref{eq:packet-criterion}.

\subsection{Canonical relations}

Throughout, $\operatorname{Supp}(F)$ denotes the closed support of a
complex $F$. Thus, for a locally closed immersion
$j:U\hookrightarrow X$, one has
$\operatorname{Supp}(j_!F)\subseteq
\overline{j(\operatorname{Supp}(F))}$.

Let $a:Y\to X$ be a $C^1$ map of real analytic manifolds. At $y\in Y$, the
transpose of $da_y:T_yY\to T_{a(y)}X$ is
${}^tda_y:T^*_{a(y)}X\to T^*_yY$, characterized by
$({}^tda_y\xi)(v)=\xi(da_y(v))$ for $v\in T_yY$. On
$Y\times_XT^*X$, define $a_d(y,\xi)=(y,{}^tda_y\xi)\in T^*Y$ and
$a_\pi(y,\xi)=(a(y),\xi)\in T^*X$.
We use the following results of Kashiwara--Schapira.

\begin{proposition}\label{prop:ks}
Let $F\in D^b_c(Y)$ and $G\in D^b_c(X)$.
\begin{enumerate}[label=\textup{(\roman*)}]
  \item If $a$ is proper on $\operatorname{Supp}(F)$, then
  \begin{equation}\label{eq:proper-ss}
    \SSupp(Ra_!F)
    \subseteq
    a_\pi\bigl(a_d^{-1}\SSupp(F)\bigr).
  \end{equation}
  If $a$ is a closed immersion, equality holds.
  \item If $a$ is a smooth submersion, then
  \begin{equation}\label{eq:smooth-ss}
    \SSupp(a^{-1}G)
    =a_d\bigl(a_\pi^{-1}\SSupp(G)\bigr).
  \end{equation}
\end{enumerate}
\end{proposition}

\begin{proof}
For \textup{(i)}, properness on $\operatorname{Supp}(F)$ identifies
$Ra_!F$ with $Ra_*F$, so the assertion, including equality for a closed
immersion, is \cite[Proposition~4.1.1(i),(ii)]{KS85}.  Assertion
\textup{(ii)} is \cite[Proposition~4.1.2(i)]{KS85}.
When $a$ is a submersion, $da_y$ is surjective and hence its transpose
${}^tda_y$, and therefore $a_d$, is injective.
\end{proof}

For an open immersion $j:V\hookrightarrow X$, \eqref{eq:smooth-ss} gives the locality formula
\begin{equation}\label{eq:locality}
  \SSupp(j^{-1}F)=\SSupp(F)\cap T^*V.
\end{equation}

\subsection{Perverse d\'evissage at a regular conormal point}

Let $H$ be a complex pro-algebraic group whose action on the smooth
complex variety $X$ factors through a complex algebraic quotient
$H\twoheadrightarrow H_0$, and assume that there are finitely many
orbits. In our applications, $(H,H_0)$ is
$(\vG^{\alg},\vG)$ or $(\vK^{\alg},\vK)$.
The $H$- and $H_0$-orbits coincide and form a Whitney stratification,
taking connected components as strata when needed
\cite[Section~2.3, p.~3582]{FR18}.
All equivariant coefficient systems and stabilizer actions below
are retained for $H$.
Fix a regular conormal datum $\nu=(x,\xi)$ over an $H$-orbit $S$ of
complex dimension $s$. Choose the normal slice and Morse pair
$J_\nu\supset K_\nu$ of \cite[(24.10)(a)]{ABV92}. For every
$C\in D^b_{c,H}(X)$, define the normalized Morse complex by
\begin{equation}\label{eq:normal-morse}
 \mu_\nu(C):=R\Gamma(J_\nu,K_\nu;C)[-s]\in D^b(\mathrm{Vect}_{\CC}).
\end{equation}
Here $R\Gamma(J_\nu,K_\nu;C)$ is the relative hypercohomology complex
of $C$ on the Morse pair; explicitly, it is
$\operatorname{Cone}(R\Gamma(J_\nu;C)\to R\Gamma(K_\nu;C))[-1]$.
This is the construction of \cite[Definition~24.11]{ABV92} for bounded
constructible complexes. As $\nu$ varies in a connected component of the
regular conormal locus, the groups $H^i\mu_\nu(C)$ form a local system. In
particular, different choices of normal slices and Morse pairs give locally
canonically isomorphic Morse groups. For a perverse sheaf $P$,
\cite[(24.10)(b),(c) and
Theorem~24.8(c)]{ABV92} gives
\begin{equation}\label{eq:morse-perverse}
 H^r\mu_\nu(P)=0\quad(r\neq0),
 \qquad
 H^0\mu_\nu(P)\simeq\For\chi_\nu^{\mic}(P).
\end{equation}
Here $H_\nu=\operatorname{Stab}_H(x,\xi)$ and
$A_\nu^{\mathrm{stab}}=H_\nu/(H_\nu)_0$ is the pointwise stabilizer
component group, where $(H_\nu)_0$ is the identity component of $H_\nu$. For every
regular $\nu$, $\chi_\nu^{\mic}(P)$ denotes the fiber at $\nu$ of the
$H$-equivariant local system $Q^{\mic}(P)$ of
\cite[Theorem~24.8]{ABV92}. Its isotropy action factors through
$A_\nu^{\mathrm{stab}}$ by \cite[Lemma~7.3(d)]{ABV92}, and
$\For:\Rep(A_\nu^{\mathrm{stab}})\to\mathrm{Vect}_{\CC}$ is the
forgetful functor for finite-dimensional continuous complex
representations.

To identify this pointwise group with the ABV equivariant
micro-fundamental group, choose an $H_x$-invariant nonempty open subset
$U_{S,x}\subset (T^*_{S,x}X)_{\mathrm{reg}}$ as in
\cite[Lemma~24.3(f) and Definition~24.7]{ABV92} and require
$\xi\in U_{S,x}$. Then $A_\nu^{\mathrm{stab}}=A_\nu^{\mic}$ is a
realization of the generic group $A_S^{\mic}$.
The concentration in \eqref{eq:morse-perverse} and exactness of this
fiber functor, hence Lemma~\ref{lem:devissage} below, hold at every
regular point.
The Arthur covectors $\nu_a$ already lie in the open conormal orbits
\cite[Proposition~3.7.1]{AIMV26}, so they belong to these invariant
generic open sets. Their fiber representations agree with
$\chi_{\psi_a}^{\mic}$ under the Arthur-component-group identification
of \cite[Proposition~22.9(e) and Definition~24.15]{ABV92}.
If $\nu=(x,0)$ lies
over an open orbit, the normal slice is a point and
\eqref{eq:normal-morse} is the shifted stalk $i_x^{-1}C[-s]$.

\begin{lemma}\label{lem:devissage}
For every $C\in D^b_{c,H}(X)$ and every $j\in\ZZ$, there is a natural isomorphism
\begin{equation}\label{eq:devissage}
 H^j\mu_\nu(C)
 \simeq
 \For\chi_\nu^{\mic}({}^p\mathcal H^j(C)).
\end{equation}
\end{lemma}

\begin{proof}
Apply the triangulated functor $\mu_\nu$ to the finite perverse tower of
$C$. The associated spectral sequence is
$E_2^{r,s}=H^r\mu_\nu({}^p\mathcal H^s(C))\Longrightarrow
H^{r+s}\mu_\nu(C)$. By \eqref{eq:morse-perverse}, only the row $r=0$ can
be nonzero. The spectral sequence therefore degenerates, and for each total
degree $j$ the induced filtration has the single nonzero graded piece
$H^0\mu_\nu({}^p\mathcal H^j(C))\simeq
\For\chi_\nu^{\mic}({}^p\mathcal H^j(C))$.
The edge morphism gives the natural isomorphism \eqref{eq:devissage}.
\end{proof}

\begin{lemma}\label{lem:morse-detects-ss}
For $C\in\mathcal D_g$, if $\mu_{\nu_g}(C)\neq0$, then
$\nu_g\in\SSupp(C)$.
\end{lemma}

\begin{proof}
Suppose first that $e=0$, and let $S_g$ be the orbit of $x_g$. If
$S_g$ were not open, finiteness of the orbit decomposition would provide
an orbit $S'\neq S_g$ with $x_g\in\overline{S'}$. The closure is
$H$-stable, so $S_g\subset\overline{S'}$. Since the zero
section over $S'$ is contained in $T^*_{S'}X$, this would imply
\[
 (x_g,0)\in\overline{T^*_{S'}X},
\]
contradicting the regularity of $\nu_g=(x_g,0)$ in the sense of
\cite[(24.1)]{ABV92}. Hence $S_g$ is open and
$\mu_{\nu_g}(C)=i_{x_g}^{-1}C[-s]\neq0$. Thus
$x_g\in\operatorname{Supp}(C)$, equivalently
$(x_g,0)\in\SSupp(C)$.

Now suppose that $e\neq0$. Choose the adapted real function $f_g$ used
to construct $J_{\nu_g}\supset K_{\nu_g}$ in
\cite[(24.10)(a)]{ABV92}, with $f_g(x_g)=0$ and
$df_g|_{x_g}=2\operatorname{Re}(e)$, as prescribed by
Section~\ref{sec:cotangent-conventions}. We may choose its extension off
the normal slice so that $f_g|_{S_g}$ has a nondegenerate critical point
at $x_g$. Indeed, in local product coordinates, add a sufficiently small
generic real quadratic form $q$ in the tangential variables, chosen so
that the Hessian of $(f_g+q)|_{S_g}$ at $x_g$ is nondegenerate. The form
$q$ vanishes on the normal slice and has zero value and differential at
$x_g$; replacing $f_g$ by $f_g+q$ therefore changes neither
$df_g|_{x_g}$ nor the normal Morse pair $(J_{\nu_g},K_{\nu_g})$, thus we still
use the notation $f_g$.

Choose a control-data product neighborhood $U\simeq B\times N$ of $x_g$,
where $B\subset S_g$ is a contractible ball and $N$ is the normal slice.
Let $(M^+_{\mathrm{tan}},M^-_{\mathrm{tan}})$ be the Morse pair of
$f_g|_{S_g}$ at $x_g$. After shrinking $U$, the added quadratic form
depends only on $B$, and the normal pair on $N$ remains
$(J_{\nu_g},K_{\nu_g})$. The stratified Morse product theorem
\cite[Theorem~3.7, pp.~65--66, and \S6.A.1, pp.~222--223]{GM88}
identifies the local Morse pair with
\[
 \bigl(M^+_{\mathrm{tan}}\times J_{\nu_g},
 (M^-_{\mathrm{tan}}\times J_{\nu_g})
 \cup(M^+_{\mathrm{tan}}\times K_{\nu_g})\bigr).
\]
Let $\ell$ be the Morse index of $f_g|_{S_g}$ at $x_g$.
Choose a sufficiently small compact Morse neighborhood
$V=\overline B_\varepsilon(x_g)\subset U$ and $0<\delta\ll\varepsilon$
as in \cite[Lemma~10.5]{Gor21}. Its boundary is transverse to the
ambient and normal-slice strata. Stratify $V$ by the intersections
of the ambient strata with its interior and boundary, choosing the
data so that $x_g$ is the only critical point
of $f_g$ on these strata with critical value in $[-\delta,\delta]$.
Put $V_{\le t}=V\cap\{f_g\le t\}$ and
$V_{<0}=V\cap\{f_g<0\}$.
The local supported-cohomology description
\cite[\S10.8, pp.~45--46]{Gor21} gives the first two isomorphisms
below; local Morse excision and the comparison with constructible
coefficients \cite[Theorems~5.3--5.4, especially part~(iv) of the proof
of Theorem~5.4, pp.~153--157]{Hamm15} give the last. For every $k\in\ZZ$,
\begin{equation}\label{eq:local-normal-morse}
\begin{aligned}
 H^k\!\left(\bigl(R\Gamma_{\{f_g\geq0\}}C\bigr)_{x_g}\right)
 &\simeq H^k(V_{\le\delta},V_{<0};C)\\
 &\simeq H^k(V_{\le\delta},V_{\le-\delta};C)\\
 &\simeq H^{k-\ell}(J_{\nu_g},K_{\nu_g};C).
\end{aligned}
\end{equation}
In the first step, the stalk limit stabilizes for these admissible
small neighborhoods; the second uses stratified isotopy below the
critical value. The normal Morse pair is taken at sufficiently small
choices, whose cohomology is independent of these choices by
\cite[(24.10) and Definition~24.11]{ABV92}.
The coefficients in the last line are the restriction of $C$ to the
normal slice. On the controlled product used for this comparison,
part~(iv) of Hamm's proof of Theorem~5.4 identifies the coefficient complex up to
quasi-isomorphism with the pullback of this restriction.
Hamm uses nonnegative constructible complexes. Since $C$ is
cohomologically bounded, an overall shift permits a representative
vanishing in negative degrees. Relative hypercohomology commutes
with this shift, so shifting back gives
\eqref{eq:local-normal-morse} for $C$.
Since $\mu_{\nu_g}(C)\neq0$, some group in the last line is nonzero.
Hence the local cohomology complex on the left is nonzero, and
\cite[Theorem~3.1.1]{KS85} gives
$(x_g,2\operatorname{Re}(e))\in\SSupp(C)$. Under the cotangent convention
of Section~\ref{sec:cotangent-conventions}, this is precisely
$\nu_g\in\SSupp(C)$.
\end{proof}

\begin{lemma}\label{lem:no-cancellation}
Let $C\in\mathcal D_g$. Suppose a visible simple perverse sheaf
$P_g(\eta)$ has nonzero coefficient in $[C]_{K_0}$. Then
$\nu_g\in\SSupp(C)$.
\end{lemma}

\begin{proof}
Write $[C]_{K_0}=\sum_j(-1)^j[{}^p\mathcal H^j(C)]$.
A nonzero total coefficient of $P_g(\eta)$ implies that it is a
Jordan--H\"older constituent of ${}^p\mathcal H^j(C)$ for at least one $j$.
Exactness of $\chi_{\psi^N}^{\mic}$ gives
\[
 \dim\For\chi_{\psi^N}^{\mic}({}^p\mathcal H^j(C))
 =\sum_Q[{}^p\mathcal H^j(C):Q]\,
   \dim\For\chi_{\psi^N}^{\mic}(Q).
\]
All summands are nonnegative, and the summand indexed by the visible
simple $P_g(\eta)$ is positive. Hence
$\For\chi_{\psi^N}^{\mic}({}^p\mathcal H^j(C))\neq0$. By
Lemma~\ref{lem:devissage}, $\mu_{\nu_g}(C)\neq0$, and
Lemma~\ref{lem:morse-detects-ss} gives $\nu_g\in\SSupp(C)$.
\end{proof}

\section{Transition of singular support}

Fix once and for all a nondegenerate $\vG^\Gamma$-invariant symmetric
complex bilinear form $B$ on $\vfg$, as in
\cite[Proposition~3.7.1]{AIMV26}, and use the same symbol for its
restriction to $\vfr$. This restriction is nondegenerate and
$\vtheta$-invariant. It identifies $\vfr^*$ with $\vfr$; in particular,
for every parabolic $\vfp\subset\vfr$ with nilradical $\vfu$, one has
$\vfp^\perp=\vfu$. Under the convention of
Section~\ref{sec:cotangent-conventions}, twice the real part of $B$
realizes the corresponding real-cotangent identification. All
identifications of Arthur covectors with the element $e$, and all
cotangent-space identifications in Proposition~\ref{prop:propagation},
are made using this same form.

All closures of orbits and double cosets in this section are taken in the
indicated complex algebraic varieties.

\begin{lemma}\label{lem:diagonal-orbit-closure}
Let $\pi_f:\vR\to X_f=\vR/\vP_f$ be the quotient map and
$\Phi:\vR\times^{\vP_g}X_f\to X_g\times X_f$ the isomorphism
$\Phi([s,y])=(s\vP_g,sy)$. Then
\[
 \overline{\vR\cdot(x_g,x_f)}
 =\Phi\bigl(\vR\times^{\vP_g}
        \overline{\vP_g\cdot x_f}\bigr),
 \qquad
 \pi_f^{-1}(\overline{\vP_g\cdot x_f})
 =\overline{\vP_g\vP_f}.
\]
In particular, the fiber of the first closure over $x_g$ is
$\{x_g\}\times\overline{\vP_g\cdot x_f}$.
\end{lemma}

\begin{proof}
Under $\Phi$, the diagonal orbit corresponds to
$\vR\times^{\vP_g}(\vP_g\cdot x_f)$. Since
$\overline{\vP_g\cdot x_f}$ is a closed $\vP_g$-stable subvariety of
$X_f$, its associated subbundle is closed in
$\vR\times^{\vP_g}X_f$; the orbit subbundle is dense in it. This proves
the first identity and the assertion about its fiber. For the second,
put $A=\vP_g\vP_f=\pi_f^{-1}(\vP_g\cdot x_f)$. For $p\in\vP_f$,
right multiplication by $p$ is a homeomorphism and $Ap=A$; hence
$\overline A\,p=\overline{Ap}=\overline A$. Thus $\overline A$, and
therefore its complement, is right-$\vP_f$-saturated. Since $\pi_f$ is
open, that complement is the inverse image of its open image under
$\pi_f$. The closed set
$\pi_f^{-1}(\overline{\vP_g\cdot x_f})$ contains $A$, so it contains
$\overline A$. Conversely, the open image of the complement of
$\overline A$ is disjoint from $\vP_g\cdot x_f$, hence from its closure.
Taking inverse images and then complements gives
$\pi_f^{-1}(\overline{\vP_g\cdot x_f})=\overline A$.
\end{proof}

\begin{proposition}\label{prop:propagation}
For every $P\in\mathcal P_f^\psi$, one has
$\nu_g\in\SSupp(I_\psi P)\Longrightarrow\nu_f\in\SSupp(P)$.
\end{proposition}

\begin{proof}
Via the normalized equivalences of Lemma~\ref{lem:fiber-model}, we can write
$P$ also for $\operatorname{Fib}_f(P)$ and write
$I=I_{\mathrm{fib}}=(q_g)_!q_f^*$ within this proof on the flag-variety
categories. The lemma identifies the singular-support statements for
$I_\psi P$ and its fiber restriction. All sheaf operations
in this section are derived.
We first treat $e=0$. By \cite[Proposition~3.7.1]{AIMV26}, the stabilizer
orbit of $e$ is open and dense in the conormal fiber. If $e=0$, that orbit
is the singleton $\{0\}$, so the conormal fiber is zero and the
corresponding $\vK$-orbit in the flag fiber is open. When $e=0$,
Lemma~\ref{lem:three-parabolics} gives
$\overline{\vP_f\vP_g}=\vP_f$, and hence
$\vP_g\subseteq\vP_f$. Therefore $Z=X_g$,
$q_g=\operatorname{id}$, and
$I\simeq q_f^*$. The hypothesis gives
$(x_g,0)\in\SSupp(q_f^*P)$. The smooth inverse-image formula
\eqref{eq:smooth-ss} gives $\xi\in\SSupp(P)_{x_f}$ with
${}^tdq_f\xi=0$. Since $q_f$ is a submersion, ${}^tdq_f$ is injective.
Thus $\xi=0$, and $(x_f,0)\in\SSupp(P)$.

Assume henceforth that $e\neq0$. Factor $q_g=p_g\circ\iota$, where
$\iota:Z\to X_g\times X_f$ is given by
$\iota(z)=(q_g(z),q_f(z))$ and $p_g$ is the first projection. Set
$\cF=\iota_!q_f^*P$. Then $I P\simeq(p_g)_!\cF$. Since $X_f$ is
projective, $p_g$ is proper.

Suppose $\nu_g=(x_g,e)\in\SSupp((p_g)_!\cF)$. By
\eqref{eq:proper-ss}, there exists $y\in X_f$ such that
\begin{equation}\label{eq:lifted-covector}
  \bigl((x_g,y),(e,0)\bigr)\in\SSupp(\cF).
\end{equation}
Write $y=r\vP_f$.

For a $\vK$-orbit
$S\subset X_f$ and an $\vR$-orbit
$\mathcal T\subset X_g\times X_f$, the subsets
$\mathcal S_{S,\mathcal T}=(X_g\times S)\cap\mathcal T$
form the incidence decomposition used in
\cite[(4.6.5)]{AIMV26}. We verify that it is Whitney.
The $\vK$-orbit decomposition of $X_f$ and the $\vP_f$-orbit
decomposition of $X_g$ are finite algebraic orbit decompositions
(for the latter, use the Bruhat decomposition), hence Whitney
\cite[Section~2.3, p.~3582]{FR18}.
Choose a local analytic section $s:U\to\vR$ of
$\vR\to X_f$, so that $s(y)x_f=y$. The analytic isomorphism
\[
 X_g\times U\longrightarrow X_g\times U,\qquad
 (z,y)\longmapsto(s(y)^{-1}z,y)
\]
takes $\mathcal T\cap(X_g\times U)$ to
$O_{\mathcal T}\times U$, where
$O_{\mathcal T}=\{z\in X_g:(z,x_f)\in\mathcal T\}$ is a
$\vP_f$-orbit. It therefore takes the incidence pieces to
$O_{\mathcal T}\times(S\cap U)$, a product of Whitney strata.
Products and analytic changes of coordinates preserve the Whitney
conditions. The same local description gives
\[
 \overline{\mathcal S_{S,\mathcal T}}
 =(X_g\times\overline S)\cap\overline{\mathcal T},
\]
which verifies the frontier condition. If connected strata are
required, replace each $\vK$-orbit by its finitely many
$\vK^\circ$-orbits; the fibers $O_{\mathcal T}$ are connected
because $\vP_f$ is connected. This leaves all displayed conormal
unions unchanged.

The second projection
$\mathcal T\to X_f$ is a smooth surjective map of homogeneous spaces,
so $\mathcal T$ meets $X_g\times S$ transversely. Since $P$ is
constructible along the $\vK$-orbits, $q_f^*P$ is constructible along
the induced strata on $Z$. The image of $\iota$ is a union of strata,
so its extension by zero is constructible for this stratification as well.
The Whitney microsupport estimate
\cite[Proposition~4.8]{GPS23}, applied to $\iota_!q_f^*P$, gives
the following inclusion, also recorded in \cite[(4.6.5)]{AIMV26}.
Here we use
$\operatorname{Supp}(\iota_!q_f^*P)\subseteq
\overline{\vR\cdot(x_g,x_f)}$ to restrict the orbit range:
\begin{equation}\label{eq:aimv-incidence}
 \SSupp(\iota_!q_f^*P)
 \subseteq
 \bigcup_{\substack{S\in\vK\backslash X_f\\
 \mathcal T\subseteq\overline{\vR\cdot(x_g,x_f)}}}
 T^*_{\mathcal S_{S,\mathcal T}}(X_g\times X_f)
 .
\end{equation}
Here $\mathcal T$ ranges over the $\vR$-orbits contained in the indicated
diagonal-orbit closure, and each
$T^*_{\mathcal S_{S,\mathcal T}}(X_g\times X_f)$ is the ordinary conormal
bundle of the stratum.
The union of ordinary conormal bundles of a Whitney stratification
is closed by Whitney's condition~(a)
\cite[Section~2.1]{GPS23}; hence no additional conormal closures
are needed in \eqref{eq:aimv-incidence}.

For completeness, set $b=(x_g,r\vP_f)\in\mathcal S_{S,\mathcal T}$.
The fixed bilinear form identifies
$T_b^*(X_g\times X_f)=\vfu_g\oplus\operatorname{Ad}(r)\vfu_f$.
The conormal spaces of $X_g\times S$ and $\mathcal T$ at $b$ are
\[
 \{0\}\oplus(\operatorname{Ad}(r)\vfu_f)^{-\vtheta},
 \qquad
 \{(u,-u):u\in\vfu_g\cap\operatorname{Ad}(r)\vfu_f\},
\]
respectively; here $W^{-\vtheta}:=W\cap\ker(\vtheta+1)$ for a
subspace $W\subset\vfr$. Transversality makes the conormal to their
intersection the sum of these two spaces. Thus
\cite[(4.6.6)--(4.6.8)]{AIMV26} gives
\begin{equation}\label{eq:incidence-conormal-fiber}
 T^*_{\mathcal S_{S,\mathcal T},b}(X_g\times X_f)
 =\left\{(u,v)\in\vfu_g\oplus\operatorname{Ad}(r)\vfu_f:
 \begin{aligned}
 &u\in\vfu_g\cap\operatorname{Ad}(r)\vfu_f,\\
 &u+v\in(\operatorname{Ad}(r)\vfu_f)^{-\vtheta}
 \end{aligned}\right\}.
\end{equation}

By \eqref{eq:lifted-covector} and \eqref{eq:aimv-incidence}, the covector
$((x_g,r\vP_f),(e,0))$ belongs to one of these conormal bundles.
Substituting $(u,v)=(e,0)$ in \eqref{eq:incidence-conormal-fiber} gives
$e\in\vfu_g\cap\operatorname{Ad}(r)\vfu_f$ and
$e\in(\operatorname{Ad}(r)\vfu_f)^{-\vtheta}$.
In particular,
\begin{equation}\label{eq:e-in-adhuf}
 e\in\operatorname{Ad}(r)\vfu_f.
\end{equation}
Every base point of a microsupport covector lies in the closed support of the
complex. Hence \eqref{eq:lifted-covector} gives
$(x_g,r\vP_f)\in\operatorname{Supp}(\cF)$.
Since $\cF=\iota_!q_f^*P$ and
$\iota(Z)=\vR\cdot(x_g,x_f)$, one has
$\operatorname{Supp}(\cF)\subseteq\overline{\vR\cdot(x_g,x_f)}$.
Thus $(x_g,r\vP_f)\in\overline{\vR\cdot(x_g,x_f)}$.
Lemma~\ref{lem:diagonal-orbit-closure} first gives
$r\vP_f\in\overline{\vP_g\cdot x_f}$ and then
\begin{equation}\label{eq:r-in-closure}
 r\in\overline{\vP_g\vP_f}.
\end{equation}
Inversion is an algebraic automorphism of $\vR$, so
\eqref{eq:e-in-adhuf} and \eqref{eq:r-in-closure} give
$r^{-1}\in\overline{\vP_f\vP_g}$ and
$\operatorname{Ad}(r^{-1})e\in\vfu_f$.
By Lemma~\ref{lem:three-parabolics}, $r^{-1}\in\vP_f$, and therefore
$y=x_f$. We have proved
\begin{equation}\label{eq:distinguished-covector}
  \bigl((x_g,x_f),(e,0)\bigr)
  \in\SSupp(\iota_!q_f^*P).
\end{equation}

The incidence map $\iota$ is a locally closed immersion. Choose an open
neighborhood $j:V\hookrightarrow X_g\times X_f$ of $(x_g,x_f)$ such that
the restriction $\iota_V:U\hookrightarrow V$ is a
closed immersion, where $U=\iota^{-1}(V)$. By \eqref{eq:locality},
\eqref{eq:distinguished-covector} remains valid after restriction to $V$,
and $j^{-1}(\iota_!q_f^*P)\simeq(\iota_V)_*((q_f^*P)|_U)$.
Let $z_0=1(\vP_g\cap\vP_f)\in Z$. The equality for a closed immersion in
Proposition~\ref{prop:ks}(i), followed by locality
\eqref{eq:locality}, gives
\begin{equation}\label{eq:pullback-iota}
  {}^td\iota_{z_0}(e,0)
  \in\SSupp(q_f^*P)_{z_0}.
\end{equation}
Since $\iota=(q_g,q_f)$,
\begin{equation}\label{eq:qg-pullback}
  {}^td\iota_{z_0}(e,0)
  ={}^tdq_g(e).
\end{equation}

The map $q_f$ is a smooth submersion. By \eqref{eq:smooth-ss}, \eqref{eq:pullback-iota}, and \eqref{eq:qg-pullback}, there exists
$\xi\in\SSupp(P)_{x_f}$
such that
\begin{equation}\label{eq:pullback-equality}
  {}^tdq_f(\xi)
  ={}^tdq_g(e).
\end{equation}

We now identify $\xi$. Use the invariant complex bilinear form fixed above,
and twice its real part for the real cotangent bundles. At the base points,
$T^*_{x_a}X_a\simeq\vfp_a^{\perp}=\vfu_a$ for $a\in\{f,g\}$, while
$T^*_{z_0}Z\simeq\Lie(\vP_f\cap\vP_g)^{\perp}$.
Since $T_{1\vP_a}(\vR/\vP_a)=\vfr/\vfp_a$, the transpose differentials
${}^tdq_f$ and ${}^tdq_g$ are the natural inclusions
$\vfp_a^\perp\hookrightarrow
\Lie(\vP_f\cap\vP_g)^\perp$.

Write $\lambda_a=\tfrac12h+\zeta_a$. The elements $\zeta_a$
centralize the
Arthur $\mathfrak{sl}_2$, so
$[\lambda_a,e]=\tfrac12[h,e]+[\zeta_a,e]=e$.
Hence $e\in\vfu_f\cap\vfu_g$. Under the inclusions just described,
\begin{equation}\label{eq:common-e}
  {}^tdq_f(e)
  ={}^tdq_g(e).
\end{equation}
Combining \eqref{eq:pullback-equality} and \eqref{eq:common-e} yields
${}^tdq_f(\xi)={}^tdq_f(e)$.
Since $q_f$ is a submersion, ${}^tdq_f$ is injective. Therefore $\xi=e$,
and $\nu_f\in\SSupp(P)$.
\end{proof}

\section{Coherent continuation of individual packet members}

Let $\mathcal D(G^\Gamma)$ be the set of $G$-conjugacy classes of strong
real forms. For each $[\gamma]\in\mathcal D(G^\Gamma)$, choose a
representative, let $K_\gamma$ be the complexification of a maximal
compact subgroup of $G(\RR,\gamma)$, and abbreviate
$K\Pi_\lambda^z(\fg,K_\gamma)=KM_{\lambda,\mathrm{fl}}^z(\fg,K_\gamma)$.
Here $M_{\lambda,\mathrm{fl}}^z(\fg,K_\gamma)$ is the category of finite-length
projective type-$z$ $(\fg,K_\gamma)$-modules with infinitesimal character
$\lambda$, and $KM_{\lambda,\mathrm{fl}}^z(\fg,K_\gamma)$ is its
Grothendieck group.
Then we have a natural identification
\begin{equation}\label{eq:strong-real-form-sum}
 K\Pi_\lambda^z(G^\Gamma)
 \xrightarrow{\ \sim\ }
 \bigoplus_{[\gamma]\in\mathcal D(G^\Gamma)}
 K\Pi_\lambda^z(\fg,K_\gamma).
\end{equation}
The realization \eqref{eq:strong-real-form-sum} depends on the chosen
representatives and maximal compact subgroups. Conjugation relates any two
choices, and Lemma~\ref{lem:conjugacy-naturality} shows that this transport
commutes with continuation. Since $\lambda_g$ is in the good range for
$L_N$ and $\lambda_f\in\lambda_g+X^*(L_N)$, each
$M\in K\Pi_{\lambda_g}^z(\fg,K_\gamma)$ determines a unique
$L_N$-coherent family $\Theta_{M,\gamma}$ with
$\Theta_{M,\gamma}(\lambda_g)=M$; see
\cite[Proposition~2.10.2]{AIMV26}. On the summand indexed by $\gamma$,
\cite[Definition~2.10.3]{AIMV26} defines coherent continuation by
\[
 \begin{aligned}
 T_{L_N,\gamma}(\lambda_g\to\lambda_f):
 K\Pi_{\lambda_g}^z(\fg,K_\gamma)
 &\longrightarrow K\Pi_{\lambda_f}^z(\fg,K_\gamma),\\
 M&\longmapsto\Theta_{M,\gamma}(\lambda_f).
 \end{aligned}
\]
Under
\eqref{eq:strong-real-form-sum}, the global continuation is represented summandwise by
\begin{equation}\label{eq:summandwise-continuation}
 \bigoplus_{[\gamma]\in\mathcal D(G^\Gamma)}
 T_{L_N,\gamma}(\lambda_g\to\lambda_f).
\end{equation}
Thus continuation preserves summands on strong real forms.

\begin{proposition}\label{prop:continuation}
For every $\tau_g\in\Pi^z(G^\Gamma)^{\mic}_{\psi^N}$,
$[T\tau_g]\subseteq\Pi^z(G^\Gamma)^{\mic}_{\psi}$.
\end{proposition}

\begin{proof}
Write $\tau_g=\pi_g(\eta)$, and let $\pi_f(\xi)$ be an irreducible having
nonzero coefficient in $T\pi_g(\eta)$. By \eqref{eq:vogan},
$\langle T\pi_g(\eta),P_f(\xi)\rangle_f\neq0$.
Since $P_g(\eta)$ is visible at the Arthur covector, its support lies
in $\mathcal X_g^\psi$: a simple perverse sheaf is supported on one
connected component, and its microstalk vanishes on all other components.
The full representation and sheaf Grothendieck groups decompose into
orthogonal summands over the $Y$-blocks. Global convolution preserves
these blocks. Hence \eqref{eq:duality} first forces $P_f(\xi)$ to lie
in the same $Y$-block as the Arthur point.

We must also locate its connected component. If $j_f^*P_f(\xi)=0$,
then \eqref{eq:component-convolution} implies $j_g^*IP_f(\xi)=0$.
Because $P_g(\eta)$ is supported on $\mathcal X_g^\psi$, the diagonal
pairing would then give
$\langle\pi_g(\eta),[IP_f(\xi)]_{K_0}\rangle_g=0$, contradicting
\eqref{eq:duality}. Thus $P_f(\xi)$ is supported on
$\mathcal X_f^\psi$. Set $P_f^\psi(\xi)=j_f^*P_f(\xi)$ and
$P_g^\psi(\eta)=j_g^*P_g(\eta)$; these are simple perverse sheaves
on the selected components.

By \eqref{eq:duality},
$0\neq\langle T\pi_g(\eta),P_f(\xi)\rangle_f
=\langle\pi_g(\eta),[I P_f(\xi)]_{K_0}\rangle_g$.
The $g$-side pairing is diagonal, so $P_g(\eta)$ has nonzero coefficient
in $[IP_f(\xi)]_{K_0}$. Open and closed restriction is perverse
$t$-exact, and \eqref{eq:component-convolution} identifies its restriction
with $I_\psi P_f^\psi(\xi)$. Therefore the visible simple
$P_g^\psi(\eta)$ has the same nonzero coefficient in
$[I_\psi P_f^\psi(\xi)]_{K_0}$.
Applying Lemma~\ref{lem:no-cancellation} on this direct summand gives
$\nu_g\in\SSupp(I_\psi P_f^\psi(\xi))$.
Proposition~\ref{prop:propagation} gives
$\nu_f\in\SSupp(P_f^\psi(\xi))$. By locality under $j_f$, this is
$\nu_f\in\SSupp(P_f(\xi))$, so the packet criterion
\eqref{eq:packet-criterion} gives
$\pi_f(\xi)\in\Pi^z(G^\Gamma)^{\mic}_{\psi}$.
\end{proof}

For the argument on coherent families, we must realize the two nonzero
strong-real-form summands of extended induction on one compatible chain
of module categories. The induction maps themselves are already defined
in \cite[(2.6.1)]{AIMV26}.

\begin{lemma}\label{lem:adapted-cartan}
Let $G_0$ be a real reductive group.
\begin{enumerate}[label=\textup{(\roman*)}]
 \item If $P_0=L_0N_0$ is a real parabolic with a specified real Levi
 factor, there is a Cartan involution $\theta$ of $G_0$ such that
 $\theta(L_0)=L_0$ and $\theta(\mathfrak p)=\mathfrak p^{\mathrm{op}}$.
 \item Let $\sigma$ be the real conjugation on the complexification of
 $G_0$. If $Q=LU$ is a parabolic such that
 $\sigma(Q)=Q^{\mathrm{op}}$ with common Levi $L$, there is a Cartan
 involution $\theta$ commuting with $\sigma$ such that
 $\theta(L)=L$ and $\theta(Q)=Q$.
\end{enumerate}
Any two Cartan involutions of a real reductive group are conjugate by an
inner automorphism from the identity component.
\end{lemma}

\begin{proof}
For \textup{(i)}, begin with any Cartan involution. The Langlands
decomposition of $P_0$ supplies a $\theta$-stable real Levi factor and an
opposite nilradical. Any two real Levi factors of $P_0$ are conjugate by
$N_0$, so conjugating $\theta$ by an element of $N_0$ gives the specified
$L_0$.

For \textup{(ii)}, choose an integral grading element
$H_Q\in\mathfrak z(\mathfrak l)$ for $Q$. Since $\sigma(Q)=Q^{\mathrm{op}}$,
the element $-\sigma(H_Q)$ lies in the same open grading chamber as $H_Q$.
The integral element $H_Q-\sigma(H_Q)$ therefore defines the same parabolic
and satisfies $\sigma(H_Q-\sigma(H_Q))=-(H_Q-\sigma(H_Q))$; replace $H_Q$
by this element. Then $iH_Q\in\mathfrak l_0$ generates a compact central
torus $T_c\subset L_0$. Choose a Cartan involution $\theta_L$ of $L_0$.
Its fixed maximal compact subgroup contains $T_c$, since $T_c$ is central,
so $\theta_L(H_Q)=H_Q$. By
\cite[Theorem~3.13(1)(b)]{AT18}, $\theta_L$ extends to a Cartan involution
$\theta$ of $G_0$ preserving $L_0$. It fixes $H_Q$ and therefore preserves
the nonnegative grading parabolic $Q$. The conjugacy assertion is
\cite[Theorem~3.13(1)(a)]{AT18}. This is the numbering of the
published version; the same theorem is numbered~3.12 in
\href{https://arxiv.org/abs/1611.05956v2}{arXiv:1611.05956v2}.
\end{proof}

\begin{lemma}\label{lem:fixed-pair}
Let $\delta_2\in M_2^\Gamma\setminus M_2$ be a strong real form. Assume
that its images $\delta_1\in M_1^\Gamma$ and
$\delta_G\in G^\Gamma$ are again strong real forms. Equivalently,
$\delta_2^2\in Z(M_1)$ and $\delta_1^2\in Z(G)$, where $\delta_1$ and
$\delta_G$ denote the successive images of
$\delta_2$.
There are compatible choices of maximal compact subgroups whose
complexifications satisfy $K_{2,\delta_2}=K_{1,\delta_1}\cap M_2$ and
$K_{1,\delta_1}=K_{\delta_G}\cap M_1$.
On the corresponding Grothendieck groups of projective type-$z_2$,
type-$z_1$, and type-$z$ modules for
$(\fm_2,K_{2,\delta_2})$, $(\fm_1,K_{1,\delta_1})$, and
$(\fg,K_{\delta_G})$, respectively, the nonzero restrictions of the two maps in
\cite[(2.6.1)]{AIMV26} are, respectively, the cohomological induction
$R_{\fq_2}^{\fm_1}$ and the real-parabolic operator
$R_{\fq_1}^{\fg}$ with convention \eqref{eq:real-induction-normalization}.
\end{lemma}

\begin{proof}
The extended Levi inclusions in the Jordan decomposition are
$M_2^\Gamma\subset M_1^\Gamma\subset G^\Gamma$;
see \cite[Definition~4.1.1(5),(7)]{AIMV26}. Since $\delta_2$ is a strong
real form, the required finite-order condition is already satisfied.
Its image in $M_1^\Gamma$ is strong exactly when
$\delta_2^2\in Z(M_1)$, and the next image in $G^\Gamma$ is strong
exactly when $\delta_1^2\in Z(G)$. Formula \cite[(2.6.1)]{AIMV26}
sends a summand to zero when the corresponding centrality condition
fails. Moreover, $Q_2$ is $\theta$-stable and $Q_1$ is real in the
extended-group sense.
Thus \cite[Lemma~2.6.1]{AIMV26} gives
$\operatorname{Ad}(\delta_1)Q_2=Q_2^{\mathrm{op}}$ and
$\operatorname{Ad}(\delta_G)Q_1=Q_1$.

Lemma~\ref{lem:adapted-cartan} allows us to choose the maximal compact
subgroups simultaneously as follows.
Start with a Cartan involution $\theta_G^0$ of
$G(\RR,\delta_G)$ adapted to the real parabolic $Q_1(\RR)$; it preserves
its Levi subgroup $M_1(\RR,\delta_1)$ and satisfies
$\theta_G^0(\fq_1)=\fq_1^{\mathrm{op}}$. Here
$\operatorname{Ad}(\delta_1)Q_2=Q_2^{\mathrm{op}}$ is the extended-group
condition. Part~\textup{(ii)} of the lemma supplies a Cartan
involution $\theta_1$ of $M_1(\RR,\delta_1)$ that preserves $Q_2$ and
$M_2$, so $\theta_1(\fq_2)=\fq_2$. Any two Cartan
involutions of $M_1(\RR,\delta_1)$ are conjugate
by $M_1(\RR,\delta_1)$. Hence, after replacing $\theta_G^0$ by
$\theta_G=\operatorname{Ad}(m)\circ\theta_G^0\circ
\operatorname{Ad}(m^{-1})$ for some $m\in M_1(\RR,\delta_1)$,
its restriction to $M_1$ is $\theta_1$. This replacement still preserves
$M_1$ and remains adapted to $Q_1$, since $M_1$ normalizes $Q_1$ and
$Q_1^{\mathrm{op}}$. If $K_{\delta_G}$ is the complexification of the
maximal compact subgroup fixed by $\theta_G$, set
$K_{1,\delta_1}=K_{\delta_G}\cap M_1$ and
$K_{2,\delta_2}=K_{1,\delta_1}\cap M_2$.
These choices give the $(\fm_i,K_{i,\delta_i})$- and
$(\fg,K_{\delta_G})$-module categories. On this fixed chain, the two nonzero summands of
\cite[(2.6.1)]{AIMV26} are the Euler-characteristic cohomological
induction $R_{\fq_2}^{\fm_1}$ for the $\theta_1$-stable parabolic
$\fq_2$, followed by the real-parabolic operator
$R_{\fq_1}^{\fg}$ for $Q_1(\RR)$ with
convention \eqref{eq:real-induction-normalization}.
\end{proof}

\begin{lemma}\label{lem:conjugacy-naturality}
Let $g\in G$ satisfy $g\delta_Gg^{-1}=\gamma$, and put
$H^g=gHg^{-1}$ and $L_N^g=gL_Ng^{-1}$.
In \eqref{eq:strong-real-form-sum}, choose
$K_\gamma=gK_{\delta_G}g^{-1}$.
For $a\in\{f,g\}$, define
$\lambda_a^g=\lambda_a\circ\operatorname{Ad}(g^{-1})\in
(\operatorname{Lie}H^g)^*$; conjugation identifies
$X^*(H)$ with $X^*(H^g)$ and $X^*(L_N)$ with $X^*(L_N^g)$.
Transport of module structures by $\operatorname{Ad}(g)$ gives isomorphisms
$C_g:K\Pi_{\lambda_a}^z(\fg,K_{\delta_G})\to
K\Pi_{\lambda_a^g}^z(\fg,K_\gamma)$ for $a\in\{f,g\}$.
These isomorphisms preserve projective type and satisfy
\begin{equation}\label{eq:conjugacy-naturality}
 C_g\circ T_{L_N,\delta_G}(\lambda_g\to\lambda_f)
 =
 T_{L_N^g,\gamma}(\lambda_g^g\to\lambda_f^g)\circ C_g.
\end{equation}
Under the conjugacy identification of
\cite[Definition~2.4.2]{AIMV26}, the operator on the right is the
$\gamma$-summand denoted by
$T_{L_N,\gamma}(\lambda_g\to\lambda_f)$ in
\eqref{eq:summandwise-continuation}.
\end{lemma}

\begin{proof}
Since $G$ is connected, choose a path $g_t$ from $1$ to $g$. The maps
$x\mapsto g_txg_t^{-1}$ give a based homotopy from the identity to
$\operatorname{Int}(g)$. Equivalently, inner conjugation acts trivially on
$\pi_1(G)$. Consequently, on every finite cover in the canonical-cover
system, $\operatorname{Int}(g)$ has a unique lift fixing the identity above
$1$. These lifts are compatible and fix the covering kernels pointwise;
passing to the inverse limit gives the canonical identification used in
\cite[Definitions~2.4.2--2.4.3]{AIMV26}. The corresponding kernel character
and the projective type $z$ are unchanged, so each $C_g$ is an isomorphism
of the displayed type-$z$ groups.
The inner automorphism $\operatorname{Ad}(g)$ fixes
$Z(U(\fg))$ pointwise. Moreover, the transport of any finite-dimensional
algebraic $G$-module by $\operatorname{Ad}(g)$ is isomorphic to the
original module. Hence $C_g$ preserves the tensor identities defining
classical coherent families. It carries a classical family on
$\lambda_g+X^*(H)$ to one on
$\lambda_g^g+X^*(H^g)$ and carries its restriction on $L_N$ to the
corresponding restriction on $L_N^g$. Good range is invariant under this
transport, so the uniqueness in
\cite[Proposition~2.10.2]{AIMV26} and the definition of continuation give
\eqref{eq:conjugacy-naturality}.
Under the conjugacy identification of
\cite[Definition~2.4.2]{AIMV26}, the transported Cartan, Levi, and weights
are identified with the original abstract data. Thus the operator on the
right of \eqref{eq:conjugacy-naturality} is the $\gamma$-summand
denoted by $T_{L_N,\gamma}(\lambda_g\to\lambda_f)$ in
\eqref{eq:summandwise-continuation}.
\end{proof}

\begin{lemma}\label{lem:coherent-family}
Fix a strong real form $\delta_2$ of $M_2^\Gamma$, and let
$\sigma\in\Pi^{z_2}(M_2^\Gamma)^{\mic}_{\psi_2}$
be carried by the $\delta_2$-summand. Put
$\sigma^N=\sigma\otimes\chi^{2N}$. Then
$T(\widetilde R_N\sigma^N)=\widetilde R_0\sigma$ in
$K\Pi^z(G^\Gamma)$.
\end{lemma}

\begin{proof}
Choose an irreducible module $V$ representing $\sigma^N$ on the
$\delta_2$-summand, and write $\sigma^N=(\delta_2,V)$. Tensoring by an
algebraic character does not change the strong real form. Let $\delta_1$
and $\delta_G$ denote the
images of $\delta_2$ in $M_1^\Gamma$ and $G^\Gamma$ whenever those images
are again strong real forms.

For $\mu\in X^*(M_2)$, let $\CC_\mu$ denote the corresponding algebraic
character of $M_2$. Its pullback to the canonical cover is trivial on the
covering kernel, so tensoring by $\CC_\mu$ preserves projective type
$z_2$. We shall extend the resulting family from
$\lambda_{2,g}+X^*(M_2)$ to $\lambda_{2,g}+X^*(H)$ as an abstract
coherent family.

The inverse character $\chi^{-2N}$ determines
$\mu_0=-2N\mu_\chi=-N\kappa\in X^*(M_2)\cap X^*(L_N)$, with
$\CC_{\mu_0}=\chi^{-2N}$.
Regard $X^*(M_2)$ as a sublattice of the common ambient weight lattice
$X^*(H)$. The first induction family is prescribed on this sublattice and
will then be extended to a classical coherent family on
$\lambda_{2,g}+X^*(H)$.
Then
\begin{equation}\label{eq:source-target-shift}
 \lambda_{2,g}+\mu_0=\lambda_{2,f},
 \qquad
 \lambda_g+\mu_0=\lambda_f,
 \qquad
 \sigma^N\otimes\CC_{\mu_0}\simeq\sigma.
\end{equation}
The first two identities follow from \eqref{eq:translation}, and the third
from $\sigma^N=\sigma\otimes\chi^{2N}$. Thus both the induction family and
$T$ may be evaluated at $\mu_0$.

The induction on extended groups is defined separately on each 
summand of strong real form. A summand is sent to zero unless its strong
real form remains strong in the larger extended group. Otherwise the two
stages are the cohomological and real-parabolic operators in the
conventions above. The behavior of coherent families is given by
\cite[Proposition~2.10.4 and Proposition~2.8.4]{AIMV26}.
We work throughout in Grothendieck groups. For $\mu\in X^*(M_2)$, put
$v_\mu=[V\otimes\CC_\mu]$ and, when $\delta_2^2\in Z(M_1)$, put
$w_\mu=R_{\fq_2}^{\fm_1}v_\mu$. Thus $w_\mu$ is a Grothendieck class. The first stage sends
$(\delta_2,v_\mu)$ to $(\delta_1,w_\mu)$ when
$\delta_2^2\in Z(M_1)$ and to zero otherwise. When it is nonzero, the
second stage sends $(\delta_1,w_\mu)$ to
$(\delta_G,R_{\fq_1}^{\fg}w_\mu)$ when $\delta_1^2\in Z(G)$ and to zero
otherwise.
The choice of branch depends only on $\delta_2$, not on $\mu$. Since
$\widetilde R_N\sigma^N=\alpha_{\sigma,N}\tau_g\neq0$ at $\mu=0$, both
centrality conditions hold. The first branch therefore applies for every
$\mu$, although an individual induced class may vanish. We work in the
compatible module categories of Lemma~\ref{lem:fixed-pair}; all induced
classes lie in the indicated
$(\mathfrak m_i,K_{i,\delta_i})$- and
$(\mathfrak g,K_{\delta_G})$-module Grothendieck groups. If the
representative chosen in \eqref{eq:strong-real-form-sum} is $\gamma$, we
transport the resulting target coherent family at the end using
\eqref{eq:conjugacy-naturality}.

For $\mu\in X^*(M_2)$, put
$\Theta_1(\lambda_{2,g}+\mu)=w_\mu=R_{\fq_2}^{\fm_1}v_\mu$ in
$K\Pi^{z_1}(\mathfrak m_1,K_{1,\delta_1})$.
The projective types at the first stage are $z_2\to z_1$ by
\eqref{eq:projective-type-ledger}. By \cite[Proposition~2.10.4]{AIMV26},
$\Theta_1$ is the restriction to the $M_2$-character lattice of
a classical coherent family for $(\mathfrak m_1,K_{1,\delta_1})$. Choose
such an extension $\Theta_1^{\mathrm{cl}}$ of projective type $z_1$ on
$\lambda_{2,g}+X^*(H)$. By \cite[Proposition~2.8.4]{AIMV26}, pointwise
application of $R_{\fq_1}^{\fg}$ gives a classical coherent family for
$(\mathfrak g,K_{\delta_G})$ of projective type $z$. The type change is
$z_1\to z$ by \eqref{eq:projective-type-ledger}, and the infinitesimal
character $\lambda_{2,g}+\mu$ maps to $\lambda_g+\mu$. Denote this family by
the formula
$\Theta_2^{\mathrm{cl}}(\lambda_g+\mu)=R_{\fq_1}^{\fg}
\Theta_1^{\mathrm{cl}}(\lambda_{2,g}+\mu)$ for $\mu\in X^*(H)$.
On the $M_2$-character lattice its values are
$\Theta_2(\lambda_g+\mu)=R_{\fq_1}^{\fg}
\Theta_1(\lambda_{2,g}+\mu)$ in
$K\Pi^z(\mathfrak g,K_{\delta_G})$.
Its values on $X^*(M_2)$ do not depend on the chosen classical extension.

It remains to incorporate Kottwitz signs. On these fixed summands the
source and target Kottwitz operators act by
the constants $\epsilon_2=e(M_2(\RR,\delta_2))$ and
$\epsilon_G=e(G(\RR,\delta_G))$, both in $\{\pm1\}$. Therefore
$\widetilde R(\sigma^N\otimes\CC_\mu)=
\epsilon_G\epsilon_2\Theta_2(\lambda_g+\mu)$ for every
$\mu\in X^*(M_2)$.
For the member $\sigma^N$ in \eqref{eq:good-range}, this gives
$\alpha_{\sigma,N}=\epsilon_G\epsilon_2\beta_{\sigma,N}$. Multiplication by
$\epsilon_G\epsilon_2$ preserves the coherent-family identities, so
$\Theta^{\mathrm{cl}}:=\epsilon_G\epsilon_2\Theta_2^{\mathrm{cl}}$
is a classical coherent family for $(\mathfrak g,K_{\delta_G})$ on
$\lambda_g+X^*(H)$. For every $\mu\in X^*(M_2)$ it
satisfies
\begin{equation}\label{eq:classical-extension-values}
 \Theta^{\mathrm{cl}}(\lambda_g+\mu)
 =
 \widetilde R(\sigma^N\otimes\CC_\mu).
\end{equation}
Restricting this classical family to the auxiliary character lattice
gives, by \cite[Definition~2.10.1]{AIMV26}, the $L_N$-coherent family
$\Theta^{L_N}=\Theta^{\mathrm{cl}}|_{\lambda_g+X^*(L_N)}$ in
$K\Pi^z(\mathfrak g,K_{\delta_G})$.
It is viewed through \eqref{eq:strong-real-form-sum} as the fixed
$\delta_G$-summand of $K\Pi^z(G^\Gamma)$.

Since $0,\mu_0\in X^*(L_N)$ and
$\mu_0\in X^*(M_2)$, equations
\eqref{eq:source-target-shift} and
\eqref{eq:classical-extension-values} give
$\Theta^{L_N}(\lambda_g)=\widetilde R_N\sigma^N$ and
$\Theta^{L_N}(\lambda_f)=\widetilde R_0\sigma$.
By Lemma~\ref{lem:continuation-Levi}(1), the parabolic subalgebra
$\fq_N=\mathfrak l_N\oplus\fu_N$ has Levi factor $\mathfrak l_N$, and
$\lambda_g$ is in the good range for $\fq_N$. Thus $\lambda_g$ is in the
good range for $\mathfrak l_N$ in the sense used in
\cite[Definition~2.5.3]{AIMV26}. By \cite[Proposition~2.10.2]{AIMV26},
applied in the $(\mathfrak g,K_{\delta_G})$-module category, an
$L_N$-coherent family is uniquely determined by its value at $\lambda_g$.
Thus $\Theta^{L_N}$ is independent of the chosen classical extension, and
\cite[Definition~2.10.3]{AIMV26} gives
$T_{L_N,\delta_G}(\lambda_g\to\lambda_f)
\Theta^{L_N}(\lambda_g)=\Theta^{L_N}(\lambda_f)$. Hence
$T_{L_N,\delta_G}(\lambda_g\to\lambda_f)\widetilde R_N\sigma^N
=\widetilde R_0\sigma$.
If the representative in \eqref{eq:strong-real-form-sum} is $\gamma$,
choose $g$ as above and apply $C_g$. Conjugacy naturality
\eqref{eq:conjugacy-naturality} identifies the transported continuation
with the $\gamma$-summand, and \cite[Definition~2.4.2]{AIMV26} identifies
the transported endpoint classes with the same elements of
$K\Pi^z(G^\Gamma)$. Thus \eqref{eq:summandwise-continuation} gives the
asserted identity in $K\Pi^z(G^\Gamma)$.
\end{proof}

\section{Proof of the support equality}

\begin{theorem}\label{thm:reverse}
One has
\begin{equation}\label{eq:reverse}
 \left[
 R_{\fq_1}^{\fg}R_{\fq_2}^{\fm_1}
 \Pi^{zz(\rho_{\fu_1})z(\rho_{\fu_2})}
      (M_2^\Gamma)^{\mic}_{\psi_2}
 \right]
 \subseteq
 \Pi^z(G^\Gamma)^{\mic}_{\psi}.
\end{equation}
\end{theorem}

\begin{proof}
Fix $\sigma\in\Pi^{z_2}(M_2^\Gamma)^{\mic}_{\psi_2}$.
By \eqref{eq:packet-twist},
$\sigma^N\in\Pi^{z_2}(M_2^\Gamma)^{\mic}_{\psi_2^N}$. The 
good-range statement \eqref{eq:good-range} gives a unique
$\tau_g\in\Pi^z(G^\Gamma)^{\mic}_{\psi^N}$ and a sign
$\alpha_{\sigma,N}$ such that
$\widetilde R_N\sigma^N=\alpha_{\sigma,N}\tau_g$. Applying $T$ and
Lemma~\ref{lem:coherent-family} gives
$\widetilde R_0\sigma=T(\widetilde R_N\sigma^N)
=\alpha_{\sigma,N}T\tau_g$.
By Proposition~\ref{prop:continuation},
$[T\tau_g]\subseteq\Pi^z(G^\Gamma)^{\mic}_{\psi}$. Hence
$[\widetilde R_0\sigma]\subseteq\Pi^z(G^\Gamma)^{\mic}_{\psi}$.
Since $\sigma$ is irreducible, it belongs to a single strong-real-form
summand. Write $e(M_2^\Gamma)\sigma=\epsilon_2\sigma$, where
$\epsilon_2\in\{\pm1\}$. From
$\widetilde R_0=e(G^\Gamma)R_0e(M_2^\Gamma)$ we obtain
$R_0\sigma=\epsilon_2e(G^\Gamma)\widetilde R_0\sigma$.
The operator $e(G^\Gamma)$ is diagonal with entries $\pm1$ in the
irreducible basis, and hence
$[R_0\sigma]=[\widetilde R_0\sigma]$. By definition, $R_0\sigma$ is
$R_{\fq_1}^{\fg}R_{\fq_2}^{\fm_1}\sigma$. Taking the union over all
source-packet members proves \eqref{eq:reverse}.
\end{proof}

The opposite inclusion is \eqref{eq:intro-known}, proved in
\cite[Theorem~4.6.1(iii)]{AIMV26}.

\begin{proof}[Proof of Theorem~\ref{thm:main}]
Combine Theorem~\ref{thm:reverse} with \eqref{eq:intro-known}.
\end{proof}

\enlargethispage{2\baselineskip}

\end{document}